\documentclass[11pt]{amsart}
\usepackage {enumerate}
\usepackage{setspace}
\usepackage{caption}
\usepackage{mathrsfs}
\usepackage{hyperref}
\usepackage{esint}
\usepackage{amssymb}
\usepackage{url}

\usepackage{color}
\usepackage{lipsum}

\newtheorem{theorem}{Theorem}[section]
\newtheorem{lemma}[theorem]{Lemma}

\newtheorem{proposition}[theorem]{Proposition}

\newtheorem{conjecture}[theorem]{Conjecture}
\numberwithin{equation}{section}

\theoremstyle {definition}
\newtheorem{definition}[theorem]{Definition}
\newtheorem{remark}[theorem]{Remark}

\DeclareMathOperator{\lip}{Lip}

\begin{document}
\title[Riemannian Penrose inequality without horizon]{Riemannian Penrose inequality without horizon in dimension three}
\author{Jintian Zhu}
\address{Key Laboratory of Pure and Applied Mathematics, School of Mathematical Sciences, Peking University, Beijing, 100871, People's Republic of China}
\email{zhujintian@bicmr.pku.edu.cn}
\address{Institute for Theoretical Sciences, Westlake University, 600 Dunyu Road, 310030, Hangzhou, Zhejiang, People's Republic of China}
\email{zhujintian@westlake.edu.cn}

\begin{abstract}
Based on the $\mu$-bubble method we are able to prove the following version of Riemannian Penrose inequality {\it without horizon}: if $g$ is a complete metric on $\mathbb R^3\setminus\{O\}$ with nonnegative scalar curvature, which is asymptotically flat around the infinity of $\mathbb R^3$, then the ADM mass $m$ at the infinity of $\mathbb R^3$ satisfies
$
m\geq \sqrt{\frac{A_g}{16\pi}},
$
where $A_g$ is denoted to be the area infimum of embedded closed surfaces homologous to $\mathbb S^2(1)$ in $\mathbb R^3\setminus\{O\}$. Moreover, the equality holds if and only if there is a strictly outer-minimizing minimal $2$-sphere such that the region outside is isometric to the half Schwarzschild manifold with mass $\sqrt{\frac{A_g}{16\pi}}$.
\end{abstract}

\maketitle

\section{Introduction}
A complete Riemannian $n$-manifold $(M,g)$ is said to be asymptotically flat (AF) if $n\geq 3$ and there is a compact subset $K\subset M$ such that
\begin{itemize}
\item[(i)] the complement $M-K$ has finitely many ends, each of which is diffeomorphic to $\mathbb R^n-B_1$ with $B_1=\{x\in \mathbb R^n:|x|<1\}$,
\item[(ii)] in the Euclidean coordinate chart above the metric $g$  satisfies
\begin{equation}\label{Eq: decay}
|g_{ij}-\delta_{ij}|+|x||\partial g_{ij}|+|x|^2|\partial^2 g_{ij}|=O\left(|x|^{2-n}\right)\mbox{ as }|x|\to +\infty,
\end{equation}
\item[(iii)] the scalar curvature $R(g)$ belongs to $L^1(M,g)$.
\end{itemize}
For convenience, an end $(\mathcal E,g)$ with asymptotics \eqref{Eq: decay} will be called an AF end.
 Recall that the
Arnowitt-Deser-Misner (ADM) mass (see \cite{ADM1961}) associated to each AF end $\mathcal E$ of $(M,g)$ is defined to be
$$
m(M,g,\mathcal E)=  \frac1{2(n-1)\omega_{n-1}}\lim_{\rho\to +\infty} \int_{S_\rho} (\partial_jg_{ij}-\partial_ig_{jj}) \nu^i \, d\sigma,
$$
where $\omega_{n-1}$ is the volume of the unit sphere $\mathbb S^{n-1}(1)$, $S_\rho$ is the coordinate $\rho$-sphere $\{x\in \mathbb{R}^n: |x|=\rho\}$, and $\nu$ is the outward unit normal of $S_\rho$ in $\mathbb R^n$.

From general relativity people has great interest in showing lower bounds for ADM mass of AF manifolds with nonnegative scalar curvature.  Schoen-Yau \cite{SY1979,SY1981} (as well as Witten \cite{Witten1981}) proved the famous Riemannian positive mass theorem: the ADM mass of each AF end must be nonnegative if the given AF $3$-manifold $(M,g)$ has nonnegative scalar curvature, where the ADM mass vanishes exactly when $(M,g)$ is isometric to the Euclidean $3$-space $(\mathbb R^3,g_{euc})$. 
As a more quantitive estimate for ADM mass, Bray \cite{Bray2001} verified the famous Riemannian Penrose inequality: if $(M,g)$ is an AF $3$-manifold with nonnegative scalar curvature and $\Sigma$ is the outermost {\it minimal} $2$-spheres with respect to one AF end $\mathcal E$, then we have $m(M,g,\mathcal E)\geq \sqrt{\frac{|\Sigma|_g}{16\pi}}$, where the equality holds if and only if the region outside $\Sigma$ is isometric to half Schwarzschild manifold with mass $\sqrt{\frac{|\Sigma|_g}{16\pi}}$.
(The case when $\Sigma$ is a single $2$-sphere was handled by Huisken-Ilmanen \cite{HI2001} and recently by \cite{AMMO2022,HML2022} as well). For analogy of these results in higher dimensions, the audience can refer to \cite{Schoen1989,Witten1981,SY2020,BL2009}, but here we just focus on the three dimensional case which is enough for our purpose. It is worth pointing out the fact that Huisken-Ilmanen don't require $\Sigma$ to be necessarily minimal, and their more general result can be stated as following

\begin{theorem}[ADM-Hawking mass inequality \cite{HI2001,AMMO2022,HML2022}]\label{Thm: Penrose HI}
Let $(M^3,g)$ be an AF $3$-manifold with non-empty boundary $\partial M$ and nonnegative scalar curvature. Assume
\begin{itemize}
\item $\partial M$ is connected and outer-minimizing (see Definition \ref{Defn: outer-minimizing});
\item $M$ has only one end $\mathcal E$ and satisfies $H_2(M,\partial M,\mathbf Z)=0$.
\end{itemize}
Then we have
$$
m(M,g,\mathcal E)\geq \sqrt{\frac{|\partial M|_g}{16\pi}}\left(1-\frac{1}{16\pi}\int_{\partial M} H^2\mathrm d\sigma_g\right),
$$
where $H$ is the mean curvature of $\partial M$ with respect to inward unit normal and $\mathrm d\sigma_g$ is the area element of $\partial M$ with the induced metric.
\end{theorem}

If $\partial M$ is minimal, it is called a {\it horizon} of the AF manifold $(M,g)$. Note that large coordinate spheres in each end of an AF manifold $(M,g)$ are mean-convex with respect to the outward unit normal, and so we can always find a horizon through minimizing the area functional. As a consequence, the existence of a horizon required in Riemannian Penrose inequality always holds. However, this may not be the case when we consider the following more general class of AF manifolds.
\begin{definition}[Lesourd-Unger-Yau \cite{LUY2021}]
A complete Riemannian $n$-manifold $(M,g)$ with a distinguished end $\mathcal E$ is called {\it AF with arbitrary ends} if $n\geq 3$ and $\mathcal E$ is diffeomorphic to $\mathbb R^n-B_1$, where the metric $g$ satisfies the asymptotics \eqref{Eq: decay}, and $R(g)$ belongs to $L^1(\mathcal E,g)$.
\end{definition}
\begin{remark}
Compared to AF manifolds mentioned at the beginning of this section, AF manifolds with arbitrary ends defined above allow non-AF ends other than the distinguished AF end $\mathcal E$.
\end{remark}

The research on AF manifolds with arbitrary ends was motivated by the Liouville theorem from the conformal geometry. Schoen-Yau \cite{SY1988,SY1994} reduced the Liouville theorem to some version of Riemannian positive mass theorem for asymptotically Schwarzschild manifolds with arbitrary ends, which was recently confirmed by Lesourd-Unger-Yau \cite{LUY2021} (see \cite{LLU2022,Zhu2022} for general AF cases).

For AF manifolds with arbitrary ends the existence of a horizon no longer holds automatically. In Appendix \ref{Sec: trumpet-like} we construct a trumpet-like metric $g$ on $\mathbb R^n\setminus\{O\}$ such that $(\mathbb R^n\setminus\{O\},g)$ appears to be an AF manifold with arbitrary ends and nonnegative scalar curvature, which
has a global mean-convex foliation consisting of coordinate spheres ruling out the possibility for any minimal surface. Given this fact it seems a very interesting problem to search for appropriate geometric quantities still providing a lower bound for mass. 
This paper is devoted to finding a possible answer to this problem, where our philosophy here is to view the Riemannian Penrose inequality as some kind of systolic inequality.

Our main theorem in dimension three is stated as follows.

\begin{theorem}\label{Thm: main}
Let $O$ be the origin of $\mathbb R^3$ and $g$ be a complete metric on $\mathbb R^3\setminus\{O\}$ with nonnegative scalar curvature, which is asymptotically flat around the infinity of $\mathbb R^3$. Then the ADM mass $m$ at the infinity of $\mathbb R^3$ satisfies
\begin{equation}\label{Eq: Penrose}
m\geq \sqrt{\frac{A_g}{16\pi}},
\end{equation}
where $A_g$ is given by
\begin{equation*}
A_g=\inf\left\{\mathcal H^2_g(\Sigma)\left|
\begin{array}{c}\text{$\Sigma$ is a smoothly embedded surface }\\
\text{homologous to $\mathbb S^2(1)$ in $\mathbb R^3\setminus\{O\}$.}
\end{array}\right.\right\}
\end{equation*}
Moreover, the equality holds if and only if there is a strictly outer-minimizing minimal $2$-sphere $\Sigma_h$ homologous to $\mathbb S^2(1)$ such that the region outside $\Sigma_h$ is isometric to the half Schwarzschild manifold with mass $\sqrt{\frac{A_g}{16\pi}}>0$.
\end{theorem}
The underlying manifold in our theorem limited to $\mathbb R^3\setminus\{O\}$ is due to its convenience for us to verify topological conditions of Theorem \ref{Thm: Penrose HI} in
our frequent use of ADM-Hawking mass inequality. Other reasons for working in dimension three is that we have to bound diameter for certain surface based on Gromov's band width estimate \cite{Gromov2018} to control its mean curvature, and that we also need to use an infinite volume estimate for non-compact stable minimal surfaces from \cite{GL1983}, which are only valid for surfaces in $3$-manifolds.

Despite of these technical limitations in this work, we still hope a more general version of Theorem \ref{Thm: main} to be true. Given Bray-Lee's work \cite{BL2009} it is natural to expect similar results to hold in dimensions greater than three. Combined with the systolic nature illustrated from Theorem \ref{Thm: main} we raise the following more general

\begin{conjecture}[Riemannian Penrose inequality without horizon]
Let $(M^n,g,\mathcal E)$ be an asymptotically flat manifold with arbitrary ends and nonnegative scalar curvature. Then we have
$$
m(M,g,\mathcal E)\geq \frac{1}{2}\left(\frac{A_g}{\omega_{n-1}}\right)^{\frac{n-2}{n-1}},
$$
where $\omega_{n-1}$ is the volume of the unit sphere $\mathbb S^{n-1}(1)$ and
\begin{equation*}
A_g=\inf\left\{\mathcal H^{n-1}_g(\Sigma)\left|
\begin{array}{c}\text{$\Sigma$ is a smoothly embedded hypersurface }\\
\text{homologous to $\partial \mathcal E$ in $M$.}
\end{array}\right.\right\}.
\end{equation*}
Moreover, the equality holds if and only if there is a strictly outer-minimizing minimal $(n-1)$-sphere $\Sigma_h$ homologous to $\partial \mathcal E$ such that the region outside $\Sigma_h$ is isometric to the half Schwarzschild manifold with mass $m>0$.
\end{conjecture}

In the rest of this section, let us say some words on our proof for Theorem \ref{Thm: main}. 
First we remark that when constant $A_g$ vanishes, the inequality \eqref{Eq: Penrose} is an immediate consequence of Riemannian positive mass theorem established in \cite{LLU2022,Zhu2022}.
If the constant $A_g$ does not vanish, inspired by Theorem \ref{Thm: Penrose HI} our idea is to find outer-minimizing surface with small mean curvatures such that we can apply Theorem \ref{Thm: Penrose HI} to the region outside and then obtain the desired mass lower bound up to arbitrarily small errors. Hopefully, such surface can be constructed by setting appropriate $\mu$-bubble problems (see Section \ref{Sec: proof of inequality} for details) and this is enough for a proof of mass inequality \eqref{Eq: Penrose}.

Due to the existence of small errors in the previous discussion it does not provide a proof for rigidity and we have to directly find the horizon $\Sigma_h$ rather than approximating surfaces with small mean curvatures. The appearance of minimal hypersurface in the critical case of rigidity theorems or geometric inequalities was discussed in the previous work \cite{Zhu2020} of the author, where an approximation scheme of $\mu$-bubbles was introduced. Here we just take the same idea to construct a sequence of $\mu$-bubbles $\Sigma_i$ with mean curvature $H_i\to 0$ as $i\to\infty$ and take the limit surface $\Sigma_\infty$. Along this idea the key thing is to guarantee that $\Sigma_i$ always intersects a fixed compact subset $K$. In his work \cite{Zhu2020} the author developing a fixing trick based on the topological obstruction of $\Sigma_i$ for positive scalar curvature. This trick cannot work here since $\Sigma_i$ are topological $2$-spheres, and we have to obtain the fixing phenemenon in a more geometric way.
The crucial observation is that when equality holds in the Riemannian Penrose inequality \eqref{Eq: Penrose} those appropriately constructed surfaces with smaller and smaller mean curvature cannot jump over a region with a definite amount of volume (see Lemma \ref{Lem: volume estimate} for details). Combined with the positivity of $A_g$ we are able to show that a sequence of carefully constructed $\mu$-bubbles always intersect a fixed compact subset and consequencely they converge to our desired horizon $\Sigma_h$. We include all these technical details of our proof of rigidity in Section \ref{Sec: proof of rigidity}.

\section*{Acknowledgement}
This work is supported by China Postdoctoral Science Foundation (grant no. BX2021013). The author would like to express his gratitude to Professor Gang Tian and Professor Yuguang Shi for their constant encouragements. The author also thank Dr. Yuchen Bi for many helpful discussions during his one-week stay at University of Science and Technology of China.
\section{Proof of \eqref{Eq: Penrose}}\label{Sec: proof of inequality}
In this section, we present a proof for the desired inequality \eqref{Eq: Penrose}. First we recall the following
\begin{definition}\label{Defn: outer-minimizing}
Let $(M^n,g)$ be a Riemannian manifold with non-empty boundary $\partial M$. We say that $\partial M$ is outer-minimizing if for any hypersurface $\Sigma$ enclosing a region $\Omega$ with $\partial M$ it holds
$$
\mathcal H^{n-1}_g(\Sigma)\geq \mathcal H^{n-1}_g(\partial M).
$$
If the equality holds exactly when $\Sigma=\partial M$, we say that $\partial M$ is strictly outer-minimizing.

In our discussion below, a hypersurface $\Sigma$ in $(\mathbb R^n\setminus\{O\},g)$ is said to be (strictly) outer-minimizing when $\Sigma$ is (strictly) outer-minimizing as the boundary of the region outside $\Sigma$.
\end{definition}

As mentioned before the strategy to prove \eqref{Eq: Penrose} is to construct a sequence of connected strictly outer-minimizing surfaces $\Sigma_i$  such that mean curvatures $H_i$ of $\Sigma_i$ approach zero as $i\to\infty$. Indeed we have the following
\begin{proposition}\label{Prop: sequence}
If the constant $A_g$ in Theorem \ref{Thm: main} is positive, then there is a sequence of strictly outer-minimizing spheres $\Sigma_i$ such that $H_i\to 0$ as $i\to\infty$. Moreover, their areas $\mathcal H^2_g(\Sigma_i)$ are uniformly bounded and no less than $A_g$.
\end{proposition}
\begin{proof}
The proof is based on the $\mu$-bubble theory of Gromov from his four lecture \cite{Gromov2019}. First we point out that from the asymptotic flatness there is a positive constant $r_0$ such that coordinate spheres $S_r:=\{|x|=r\}$ with $r\geq r_0$ are all mean-convex with respect to the outward unit normal. Let us denote
$$
H_0:=\min_{S_{r_0}} H_{S_{r_0}}.
$$


In the following, we are going to construct the desired $\mu$-bubble. We start with the construction of suitable functions for later use. Through a standard mollification of the signed distance function to $S_{r_0}$, it is not difficult to construct a smooth function $\rho:\mathbb R^3\setminus\{O\}\to (-\infty,+\infty)$ satisfying
\begin{itemize}
\item $\rho(x)\to -\infty$ as $x\to O$, $\rho(x)\to +\infty$ as $x\to\infty$, and $\lip\rho<1$;
\item $B_{r_0}\setminus\{O\}=\{\rho<0\}$, where $B_{r_0}$ is the Euclidean $r_0$-ball.
\end{itemize}
For suitable prescribed function we take
$$
h_{\epsilon,\beta}:\left(-\frac{4\beta}{3\epsilon},+\infty\right)\to (\epsilon,+\infty),\quad t\mapsto \epsilon\coth\left(\frac{3}{4}\epsilon t+\beta\right),
$$
where $\beta$ is a large positive constant to be determined later. Through a direct computation it is easy to check the following properties of $h_{\epsilon,\beta}$:
\begin{itemize}
\item $h_{\epsilon,\beta}(t)\to +\infty$ as $t\to -\frac{4\beta}{3\epsilon}$ and $h_{\epsilon,\beta}'< 0$;
\item $h_{\epsilon,\beta}(t)\to \epsilon$ for any fixed $t$ as $\beta\to +\infty$;
\item $h_{\epsilon,\beta}$ satisfies the equation
$$
2h_{\epsilon,\beta}'+\frac{3}{2}h_{\epsilon,\beta}^2=\frac{3}{2}\epsilon^2.
$$
\end{itemize}

Next we work with $\epsilon<H_0$. In this case, we can take $\beta$ large enough to guarantee 
$
h_{\epsilon,\beta}(0)<H_0
$.
The $\mu$-bubble problem is now set as follows. Take
$$
M_{\epsilon,\beta}=\left\{x\in \mathbb R^3\setminus\{O\}\left|\rho(x)>-\frac{4\beta}{3\epsilon}\right.\right\}\cap \bar B_{r_0}
$$
and
$$
\mathcal C_{\epsilon,\beta}=\left\{\begin{array}{c}\mbox{Caccioppoli sets $\Omega\subset \bar B_{r_0}$ containing the origin $O$}\\
\mbox{such that $\bar B_{r_0}\setminus\Omega$ has compact closure in $M_{\epsilon,\beta}$.}
\end{array}\right\},
$$
where $\bar B_{r_0}$ is denoted to be the closure of $B_{r_0}$.
Consider the functional 
$$
\mathcal A_{\epsilon,\beta}:\mathcal C_{\epsilon,\beta}\to \mathbb R,\quad \Omega\mapsto \mathcal H^2_g(\partial\Omega)+\int_{M_{\epsilon,\beta}}\chi_{\bar B_{r_0}\setminus\Omega}\,h_{\epsilon,\beta} \circ \rho\,\mathrm d \mathcal H^3_g,
$$
where $\chi_{\bar B_{r_0}\setminus\Omega}$ is the characteristic function given by
$$
\chi_{\bar B_{r_0}\setminus\Omega}(x)=\left\{
\begin{array}{cc}
1,&x\in \bar B_{r_0}\setminus\Omega;\\
0,&x\notin \bar B_{r_0}\setminus\Omega.
\end{array}\right.
$$
It follows from \cite[Proposition 2.1]{Zhu2021} of the author (see also \cite{CL2020}) that there is a smooth region $\Omega_{\epsilon,\beta}$ in $\mathcal C_{\epsilon,\beta}$ such that
\begin{equation}\label{Eq: minimizing}
\mathcal A_{\epsilon,\beta}(\Omega_{\epsilon,\beta})=\min_{\Omega\in\mathcal C_{\epsilon,\beta}}\mathcal A_{\epsilon,\beta}(\Omega).
\end{equation}

The proof is then completed by Lemma \ref{Lem: characterization} below. In detail, we take a sequence $\epsilon_i\to 0$ as $i\to +\infty$ and $\beta_i$ large enough such that $\Omega_{\epsilon_i,\beta_i}$ satisfies those properties listed in Lemma \ref{Lem: characterization}. Then we choose $\Sigma_i$ to be the strictly outer-minimizing surface $\Sigma_{\epsilon_i,\beta_i,O}$ there. Clearly we have
$$
H_i=H_{\epsilon_i,\beta_i,O}\to 0\mbox{ as }i\to+\infty,
$$
and
$$
A_g\leq \mathcal H^2_g(\Sigma_i)\leq \mathcal H^2_g(S_{r_0}).
$$
This completes the proof.
\end{proof}

\begin{lemma}\label{Lem: characterization}
We have the following properties for $\Omega_{\epsilon,\beta}$:
\begin{itemize}
\item the exterior $\Omega_{\epsilon,\beta}^e$ of $\Omega_{\epsilon,\beta}$ is connected;

\item the component $\Omega_{\epsilon,\beta,O}$ of $\Omega_{\epsilon,\beta}$ containing the origin $O$ is homeomorphic to a $3$-ball, and its boundary $\Sigma_{\epsilon,\beta,O}$ is a smoothly embedded $2$-sphere;
\item $\Sigma_{\epsilon,\beta,O}$ is strictly outer-minimizing;
\item the mean curvature $H_{\epsilon,\beta,O}$ of $\Sigma_{\epsilon,\beta,O}$ with respect to the outward unit normal satisfies $0<H_{\epsilon,\beta,O}<2\epsilon$ for $\beta$ large enough.
\end{itemize}
\end{lemma}
\begin{proof}[Proof for Lemma \ref{Lem: characterization}]
We verify these properties one by one.

For the first property we argue by contradiction. Suppose that the exterior $\Omega_{\epsilon,\beta}^e$ of $\Omega_{\epsilon,\beta}$ has more than one components,  then we know from $\Omega_{\epsilon,\beta}\subset \bar B_{r_0}$ that there exists at least one bounded component of $\Omega_{\epsilon,\beta}^e$ contained in $M_{\epsilon,\beta}$, denoted by $U_b$. Notice that the region $\Omega_{\epsilon,\beta}\cup \bar U_b$ is still in $\mathcal C_{\epsilon,\beta}$ and we can compute
\[
\begin{split}
\mathcal A_{\epsilon,\beta}(\Omega_{\epsilon,\beta}\cup \bar U_b)&=\mathcal H^2_g(\partial\Omega_{\epsilon,\beta}\setminus \partial U_b)+\int_{M_{\epsilon,\beta}}\chi_{\bar B_{r_0}\setminus(\Omega_{\epsilon,\beta}\cup \bar U_b)}\,h_{\epsilon,\beta}\circ \rho\,\mathrm d\mathcal H^3_g\\
&=\mathcal A_{\epsilon,\beta}(\Omega_{\epsilon,\beta})-\mathcal H^2_g(\partial U_b)-\int_{\bar U_b}h_{\epsilon,\beta}\circ \rho\,\mathrm d\mathcal H^3_g\\
&<\mathcal A_{\epsilon,\beta}(\Omega_{\epsilon,\beta}),
\end{split}
\]
where the last inequality comes from the facts $\mathcal H^2_g(\partial U_b)>0$ and $h_{\epsilon,\beta}>\epsilon$. This obviously contradicts to the minimizing property \eqref{Eq: minimizing}.

For the second property we just need to show that $\Sigma_{\epsilon,\beta,O}$ is a smoothly embedded $2$-sphere, with which the $3$-ball topology of $\Omega_{\epsilon,\beta,O}$ comes from Alexander's Theorem \cite[Theorem 1.1]{Hatcher2007}. As the first step, we show that $\Sigma_{\epsilon,\beta,O}$ is connected. Otherwise, from the connectedness of $\Omega^e_{\epsilon,\beta}$ and $\Omega_{\epsilon,\beta,O}$ we can construct a closed curve in $\mathbb R^3$ intersecting some component of $\Sigma_{\epsilon,\beta,O}$ only once. On the other hand, we have the simple fact $H_1(\mathbb R^3,\mathbf Z)= 0$. This yields that the algebraic intersection number of this closed curve with any closed surface should be zero, which leads to a contradiction. Now let us confirm the spherical topology of $\Sigma_{\epsilon,\beta,O}$ through the classical variation argument. From the first variation we can compute
$$
\int_{\Sigma_{\epsilon,\beta,O}}(H_{\epsilon,\beta,O}-h_{\epsilon,\beta}\circ \rho)\psi\,\mathrm d\mathcal H^2_g=0\mbox{ for any }\psi\in C_0^\infty(\Sigma_{\epsilon,\beta,O}),
$$
where $H_{\epsilon,\beta,O}$ is the mean curvature with respect to the outward unit normal. As a consequence we have $H_{\epsilon,\beta,O}=h_{\epsilon,\beta}\circ \rho$ on $\Sigma_{\epsilon,\beta,O}$. From the second variation we obtain
\begin{equation}\label{Eq: second variation}
\begin{split}
&2\int_{\Sigma_{\epsilon,\beta,O}}|\nabla \psi|^2\,\mathrm d\mathcal H^2_g\geq\\
&\int_{\Sigma_{\epsilon,\beta,O}}\left(R(g)-R_{\epsilon,\beta,O}+H_{\epsilon,\beta,O}^2+|A_{\epsilon,\beta,O}|^2+2\partial_\nu(h_{\epsilon,\beta}\circ \rho)\right)\psi^2\,\mathrm d\mathcal H^2_g
\end{split}
\end{equation}
for any $\psi\in C_0^\infty(\Sigma_{\epsilon,\beta,O})$, where $R_{\epsilon,\beta,O}$ and $A_{\epsilon,\beta,O}$ are scalar curvature and the second fundamental form of $\Sigma_{\epsilon,\beta,O}$ respectively. After taking $\psi\equiv 1$ we see
\[
\begin{split}
4\pi \chi(\Sigma_{\epsilon,\beta,O})&\geq \int_{\Sigma_{\epsilon,\beta,O}}R(g)+H_{\epsilon,\beta,O}^2+|A_{\epsilon,\beta,O}|^2+2\partial_\nu(h_{\epsilon,\beta}\circ \rho)\,\mathrm d\mathcal H^2_g\\
&\geq \int_{\Sigma_{\epsilon,\beta,O}} R(g)+\left(\frac{3}{2}h_{\epsilon,\beta}^2+2h_{\epsilon,\beta}'\right)\circ \rho\,\mathrm d\mathcal H^2_g>0,
\end{split}
\]
where we have used the facts $\lip\rho<1$ and $R(g)\geq 0$ as well as
\begin{equation}\label{Eq: ODE}
2h_{\epsilon,\beta}'+\frac{3}{2}h_{\epsilon,\beta}^2=\frac{3}{2}\epsilon^2>0.
\end{equation}
It follows that $\Sigma_{\epsilon,\beta,O}$ is a $2$-sphere.

For the third property we take $\Sigma$ to be an arbitrary closed surface that encloses a bounded region $\Omega$ containing $\Omega_{\epsilon,\beta,O}$ as a subset. Recall that there is a mean-convex foliation outside $B_{r_0}$, so the coordinate sphere $S_{r_0}$ is strictly outer-minimizing. Let us take a new surface $\Sigma_{r_0}=\partial \Omega_{r_0}$ with $\Omega_{r_0}=\Omega\cap B_{r_0}$. Then we have
$$
\mathcal H^2_g(\Sigma)-\mathcal H^2_g(\Sigma_{r_0})=\mathcal H^2_g\left(\partial(\Omega\cup B_{r_0})\right)-\mathcal H^2_g(S_{r_0})\geq 0,
$$
where the equality holds if and only if $\Omega \subset B_{r_0}$ and $\Sigma_{r_0}=\Sigma$. Since $\Omega_{r_0}$ is contained in $B_{r_0}$, the union $\Omega_{\epsilon,\beta}\cup\Omega_{r_0}$ is still in the collection $\mathcal C_{\epsilon,\beta}$. From the minimizing property of $\Omega_{\epsilon,\beta}$ we have $\mathcal A_{\epsilon,\beta}(\Omega_{\epsilon,\beta}\cup\Omega_{r_0})\geq \mathcal A_{\epsilon,\beta}(\Omega_{\epsilon,\beta})$. We can compute
\begin{equation}\label{Eq: comparison}
\mathcal H^2_g(\Sigma_{r_0}\setminus \Omega_{\epsilon,\beta})\geq \mathcal H^2_g(\partial\Omega_{\epsilon,\beta}\cap \bar\Omega_{r_0})+\int_{\Omega_{r_0}\setminus\Omega_{\epsilon,\beta}}h_{\epsilon,\beta}\circ \rho\,\mathrm d\mathcal H^3_g,
\end{equation}
where we have used the relation
$$
\mathcal H^2_g\left(\partial(\Omega_{\epsilon,\beta}\cup \Omega_{r_0})\right)=\mathcal H^2_g(\partial\Omega_{\epsilon,\beta}\setminus\bar\Omega_{r_0})+\mathcal H^2_g(\partial \Omega_{r_0}\setminus \Omega_{\epsilon,\beta}).
$$
Due to the fact $\Omega_{\epsilon,\beta,O}\subset \Omega_{r_0}$ we see $\Sigma_{\epsilon,\beta,O}\subset \bar \Omega_{r_0}$ and so
\[
\mathcal H_{g}^2(\partial\Omega_{\epsilon,\beta}\cap \bar\Omega_{r_0})\geq \mathcal H^2_g(\Sigma_{\epsilon,\beta,O}).
\]
After substituting this into \eqref{Eq: comparison} we obtain the desired inequality
\begin{equation}\label{Eq: final comparison}
\mathcal H^2_g(\Sigma)\geq \mathcal H^2_g(\Sigma_{r_0})\geq\mathcal H^2_g(\Sigma_{r_0}\setminus\Omega_{\epsilon,\beta})\geq \mathcal H^2_g(\Sigma_{\epsilon,\beta,O}).
\end{equation}
When equality holds we have $$\mathcal H^2_g(\Sigma_{r_0})=\mathcal H^2_g(\Sigma_{r_0}\setminus\Omega_{\epsilon,\beta}) \mbox{ and }
\int_{\Omega_{r_0}\setminus\Omega_{\epsilon,\beta}}h_{\epsilon,\beta}\circ \rho\,\mathrm d\mathcal H^3_g=0.
$$
Combined with the fact $h_{\epsilon,\beta}>\epsilon$ we conclude that $\Sigma_{r_0}$ is outside $\Omega_{\epsilon,\beta}$ and also $\Omega_{\epsilon,\beta,O}\subset \Omega_{r_0}\subset \Omega_{\epsilon,\beta}$, which implies $\Omega_{r_0}$ is simply the union of several components of $\Omega_{\epsilon,\beta}$. From the fact $$\mathcal H^2_g(\Sigma)=\mathcal H^2_g(\Sigma_{r_0})=\mathcal H^2_g(\Sigma_{\epsilon,\beta,O})$$ we obtain $\Sigma=\Sigma_{r_0}=\Sigma_{\epsilon,\beta,O}$.
This verifies that $\Sigma_{\epsilon,\beta,O}$ is strictly outer-minimizing. 

For the fourth property we would like to show that $\Sigma_{\epsilon,\beta,O}$ lies in a compact subset $K_\epsilon$ of $\mathbb R^3\setminus\{O\}$ independent of $\beta$, from which we can take $\beta$ large enough such that $0<H_{\epsilon,\beta,O}<2\epsilon$. First we point out that $\Sigma_{\epsilon,\beta,O}$ has its diameter bounded from above by some positive constant $D=D(\epsilon)$. Recall from \eqref{Eq: second variation} and \eqref{Eq: ODE} that we have $$\lambda_1\left(-\Delta+\frac{1}{2}R_{\epsilon,\beta,O}\right)\geq \frac{3}{4}\epsilon^2.$$
Denote $u$ to be the corresponding first eigenfunction. It is standard that $u$ is a positive smooth function on $\Sigma_{\epsilon,\beta,O}$ and so we can consider the warped metric 
$$g_{warp}=g_{\epsilon,\beta,O}+u^2\mathrm d\theta^2\mbox{ on } \Sigma_{\epsilon,\beta,O}\times \mathbb S^1,$$ where $g_{\epsilon,\beta,O}$ is the induced metric of $\Sigma_{\epsilon,\beta,O}$. A straightforward calculation shows
$$
R(g_{warp})=R_{\epsilon,\beta,O}-\frac{2\Delta u}{u}\geq \frac{3}{2}\epsilon^2>0.
$$
It follows from Gromov's band width estimate \cite{Gromov2018} that the diameter of $\Sigma_{\epsilon,\beta,O}$ is no greater than
$
D(\epsilon)=\frac{4\pi}{3\epsilon}
$. Second we argue that $\Sigma_{\epsilon,\beta,O}$ must intersect with the compact subset 
$$
V_\epsilon=\left\{-\frac{\mathcal H^2_g(S_{r_0})}{\epsilon A_g}\leq \rho \leq 0\right\}.
$$
Otherwise, from the facts $\Omega_{\epsilon,\beta,O}\subset B_{r_0}$ and $B_{r_0}\setminus\{O\}=\{\rho <0\}$ we conclude that $\bar\Omega_{\epsilon,\beta,O}$ is contained in $\{\rho <-\mathcal H^2_g(S_{r_0})/\epsilon A_g\}$. In particular, given any negative regular value $s$ of $\rho$ no less than $-\mathcal H^2_g(S_{r_0})/\epsilon A_g$ we have from \eqref{Eq: final comparison} that
$$
\mathcal H^2_g(\{\rho =s\}\setminus \Omega_{\epsilon,\beta})\geq \mathcal H^2_g(\Sigma_{\epsilon,\beta,O})\geq A_g.
$$
From the minimizing property of $\Omega_{\epsilon,\beta}$ we have $\mathcal A_{\epsilon,\beta}(\Omega_{\epsilon,\beta})\leq \mathcal A_{\epsilon,\beta}(B_{r_0})$, which reads
$$
\mathcal H^2_g(\partial\Omega_{\epsilon,\beta})+\int_{\bar B_{r_0}\setminus\Omega_{\epsilon,\beta}}h_{\epsilon,\beta}\circ \rho\,\mathrm d\mathcal H^3_g\leq \mathcal H^2_g(S_{r_0}).
$$
From the coarea formula we can compute
\[
\begin{split}
\int_{\bar B_{r_0}\setminus\Omega_{\epsilon,\beta}}h_{\epsilon,\beta}\circ \rho\,\mathrm d\mathcal H^3_g&\geq \int_{\left\{-\frac{\mathcal H^2_g(S_{r_0})}{\epsilon A_g}\leq \rho \leq 0\right\}\setminus\Omega_{\epsilon,\beta}}h_{\epsilon,\beta}\circ \rho\,\mathrm d\mathcal H^3_g\\
&\geq \int_{-\frac{\mathcal H^2_g(S_{r_0})}{\epsilon A_g}}^0\mathrm ds\int_{\{\rho=s\}\setminus\Omega_{\epsilon,\beta}}h_{\epsilon,\beta}\circ \rho\,\mathrm d\mathcal H^3_g\\
&>\mathcal H^2_g(S_{r_0}),
\end{split}
\]
where we have used $\lip\rho<1$ in the second line and $h_{\epsilon,\beta}\circ \rho>\epsilon$ in the third line. This yields $\mathcal H^2_g(\partial\Omega_{\epsilon,\beta})<0$, which is impossible. The proof is now completed by taking $K_\epsilon$ to be the closure of $D(\epsilon)$-neighborhood of $V_\epsilon$.
\end{proof}

\begin{proof}[Proof of \eqref{Eq: Penrose}]
If the constant $A_g$ is zero, then \eqref{Eq: Penrose} simply comes from the Riemannian positive mass theorem for asymptotically flat manifolds with arbitrary ends \cite{LLU2022,Zhu2022}. Next we deal with the case $A_g>0$. In this case, we can take a sequence of strictly outer-minimizing $2$-spheres $\Sigma_i$ from Proposition \ref{Prop: sequence} with $\mathcal H^2_g(\Sigma_i)\geq A_g$ and $H_i\to 0$ as $i\to +\infty$. Denote $M_i$ to be the unbounded region outside $\Sigma_i$ and clearly we have 
$$H_2(M_i,\partial M_i,\mathbf Z)=H_2(\mathbb R^3-B,\partial B,\mathbf Z)=0.$$ After applying Theorem \ref{Thm: Penrose HI} to $(M_i,g)$ we obtain 
\[
\begin{split}
m&\geq \sqrt{\frac{\mathcal H^2_g(\Sigma_i)}{16\pi}}\left(1-\frac{1}{16\pi}\int_{\Sigma_i}H_i^2\mathrm d\mathcal H^2_g\right)\\
&\geq  \sqrt{\frac{A_g}{16\pi}}-(16\pi)^{-\frac{3}{2}}\mathcal H^2_g(\Sigma_i)^{\frac{3}{2}}\|H_i\|_{L^\infty(\Sigma_i)}^2.
\end{split}
\]
Recall from Proposition \ref{Prop: sequence} that $\mathcal H^2_g(\Sigma_i)$ is uniformly bounded. After taking $i\to+\infty$ and using the fact $H_i\to 0$, we arrive at the desired Riemannian Pensore inequality
$$
m\geq \sqrt{\frac{A_g}{16\pi}}.
$$
\end{proof}
\section{Proof of rigidity}\label{Sec: proof of rigidity}
In this section we are going to prove the rigidity part in Theorem \ref{Thm: main}. The idea is trying to find a closed minimal surface enclosing the origin $O$ under the assumption
\begin{equation}\label{Eq: mass equality}
m=\sqrt{\frac{A_g}{16\pi}}.
\end{equation}

We begin with the following area estimate.
\begin{lemma}\label{Lem: area estimate}
Fix an Euclidean ball $B_{r_0}$. There are positive constants $\epsilon_0=\epsilon_0(r_0,g)$ and $\Lambda_0=\Lambda_0(r_0,g)$ such that if $\Sigma_\epsilon$ is a strictly outer-minimizing $2$-sphere in $B_{r_0}$ with constant mean curvature $0<\epsilon<\epsilon_0$ which encloses the origin $O$, then we have
$
\mathcal H^2_g(\Sigma_\epsilon)\leq A_g+\Lambda_0\epsilon^2
$.
\end{lemma}
\begin{proof}
From Theorem \ref{Thm: Penrose HI} we have
$$
\sqrt{\frac{\mathcal H^2_g(\Sigma_\epsilon)}{16\pi}}\left(1-\frac{\epsilon^2}{16\pi}\mathcal H^2_g(\Sigma_\epsilon)\right)\leq m=\sqrt{\frac{A_g}{16\pi}}. 
$$
Since $\Sigma_{\epsilon}$ is strictly outer-minimizing, we have $\mathcal H^2_g(\Sigma_{\epsilon})\leq \mathcal H^2_g(S_{r_0})$. Take
$$
\epsilon_0=\sqrt{\frac{8\pi}{\mathcal H^2_g(S_{r_0})}}\mbox{ and }\Lambda_0=\frac{1}{2\pi}A_g\mathcal H^2_g(S_{r_0}).
$$ 
Then for $0<\epsilon<\epsilon_0$ we have
$$
\mathcal H^2_g(\Sigma_\epsilon)\leq A_g\left(1-\frac{\epsilon^2}{16\pi}\mathcal H^2_g(S_{r_0})\right)^{-2}=A_g+\Lambda_0\epsilon^2,
$$
where we use the following elementary inequality
$$
(1-\alpha)^{-2}-1<8\alpha\mbox{ for }0<\alpha<\frac{1}{2}.
$$
This completes the proof.
\end{proof}
In order to find a closed minimal surface we also need the following volume control in our construction of strictly outer-minimizing surfaces.
\begin{lemma}\label{Lem: volume estimate}
Let $B_{r_0}$, $\epsilon_0$ and $\Lambda_0$ be the same as in Lemma \ref{Lem: area estimate}. If $\Sigma_\epsilon$ is a strictly outer-minimizing $2$-sphere in $B_{r_0}$ with constant mean curvature $0<\epsilon<\epsilon_0$ enclosing the origin $O$, then for any $0<\tilde \epsilon<\epsilon$ we can find a strictly outer-minimizing $2$-sphere $\Sigma_{\tilde \epsilon}$ enclosed by $\Sigma_{\epsilon}$ with constant mean curvature $\tilde \epsilon$ enclosing the origin $O$ and satisfying the volume estimate
\begin{equation}\label{Eq: volume estimate}
\mathcal H^3_g(\Omega_{\tilde\epsilon,\epsilon})\leq \frac{\Lambda_0\epsilon^2}{\tilde\epsilon},
\end{equation}
where $\Omega_{\tilde\epsilon,\epsilon}$ is the region bounded by $\Sigma_{\tilde\epsilon}$ and $\Sigma_\epsilon$.
\end{lemma}
\begin{proof}
Let us set an appropriate $\mu$-bubble problem similarly as in the proof of Proposition \ref{Prop: sequence}. From a mollification for signed distance function to the surface $\Sigma_{\epsilon}$ we can construct a smooth function $\rho_{\epsilon}:\mathbb R^3\setminus \{O\}\to (-\infty,+\infty)$ satisfying
\begin{itemize}
\item $\rho_\epsilon(x)\to -\infty$ as $x\to O$, $\rho_\epsilon(x)\to +\infty$ as $x\to\infty$, and $\lip\rho_\epsilon<1$;
\item $\Omega_{\epsilon}\setminus\{O\}=\{\rho_\epsilon<0\}$, where $\Omega_\epsilon$ is the region enclosed by $\Sigma_{\epsilon}$.
\end{itemize}
As before, we take 
$$
h_{\tilde\epsilon,\beta}:\left(-\frac{4\beta}{3\tilde\epsilon},+\infty\right)\to (\tilde\epsilon,+\infty),\quad t\mapsto \tilde\epsilon\coth\left(\frac{3}{4}\tilde\epsilon t+\beta\right),
$$
where $\beta$ is a large positive constant such that $h_{\tilde \epsilon,\beta}(0)<\epsilon$. 
Define
$$
M_{\tilde\epsilon,\beta}=\left\{x\in \mathbb R^3\setminus\{O\}\left|\rho_\epsilon(x)>-\frac{4\beta}{3\tilde\epsilon}\right.\right\}\cap \bar \Omega_\epsilon
$$
and
$$
\mathcal C_{\tilde\epsilon,\beta}=\left\{\begin{array}{c}\mbox{Caccioppoli sets $\Omega\subset \bar \Omega_{\epsilon}$ containing the origin $O$}\\
\mbox{such that $\bar \Omega_\epsilon\setminus\Omega$ has compact closure in $M_{\tilde\epsilon,\beta}$.}
\end{array}\right\}.
$$
Consider the functional 
$$
\mathcal A_{\tilde\epsilon,\beta}:\mathcal C_{\tilde\epsilon,\beta}\to \mathbb R,\quad \Omega\mapsto \mathcal H^2_g(\partial\Omega)+\int_{M_{\tilde\epsilon,\beta}}\chi_{\bar \Omega_\epsilon\setminus\Omega}\,h_{\tilde\epsilon,\beta} \circ \rho_\epsilon\,\mathrm d \mathcal H^3_g.
$$
Then we can find a smooth minimizer $\Omega_{\tilde \epsilon,\beta}$ in $\mathcal C_{\tilde \epsilon,\beta}$ of the functional $\mathcal A_{\tilde \epsilon,\beta}$. Denote $\Omega_{\tilde\epsilon,\beta,O}$ to be the component of $\Omega_{\tilde\epsilon,\beta}$ containing the origin $O$. From a similar argument as in the proof of Lemma \ref{Lem: characterization} we see that the surface $\Sigma_{\tilde \epsilon,\beta,O}=\partial \Omega_{\tilde \epsilon,\beta,O}$ is a strictly outer-minimizing $2$-sphere contained in a fixed compact subset $K_{\tilde \epsilon}$ (independent of $\beta$), of which the mean curvature satisfies
$$
\tilde \epsilon<H_{\tilde \epsilon,\beta,O}\leq \tilde \epsilon+o(1)\mbox{ as }\beta \to +\infty.
$$

For the desired surface $\Sigma_{\tilde \epsilon}$ we pick up a sequence $\beta_i\to +\infty$ and up to a subsequence the region $\Omega_{\tilde \epsilon,\beta_i,O}$ converges to a smooth minimizer $\Omega_{\tilde \epsilon}$ in
$$
\mathcal C_{\tilde\epsilon}=\left\{\begin{array}{c}\mbox{Caccioppoli sets $\Omega\subset \bar \Omega_{\epsilon}$ containing the origin $O$}\\
\mbox{such that $\bar \Omega_\epsilon\setminus\Omega$ has compact closure in $\bar\Omega_{\epsilon}\setminus\{O\}$.}
\end{array}\right\}
$$
of the following functional
$$
\mathcal A_{\tilde\epsilon}(\Omega)= \mathcal H^2_g(\partial\Omega)+\tilde \epsilon\mathcal H^3_g(\bar \Omega_{\epsilon}\setminus \Omega)\mbox{ for }\Omega\in \mathcal C_{\tilde\epsilon}.
$$
It follows from the uniform area bound $\mathcal H^2(\Sigma_{\tilde \epsilon,\beta_i,O})\leq \mathcal H^2(\Sigma_{\epsilon})$ along with the uniform bound for mean curvature $H_{\tilde\epsilon,\beta_i,O}$ as well as the compactness of $K_\epsilon$ that we have a uniform curvature estimate
$|A_{\tilde \epsilon,\beta_i,O}|\leq C$ with $C$ independent of $i$. As a result, $\partial\Omega_{\tilde \epsilon,\beta_i,O}$ converges smoothly to $\partial\Omega_{\tilde \epsilon}$ and so the boundary $\Sigma_{\tilde \epsilon}:=\partial\Omega_{\tilde \epsilon}$ is a $2$-sphere with constant mean curvature $\tilde\epsilon$. As for $\Sigma_{\tilde \epsilon,\beta,O}$ we can verify that the surface $\Sigma_{\tilde \epsilon}$ is strictly outer-minimizing.

The desired volume control comes from the strictly outer-minimizing property of $\Sigma_{\tilde \epsilon}$. Notice that from $\mathcal A_{\tilde \epsilon}(\Omega_{\tilde \epsilon})\leq \mathcal A_{\tilde \epsilon}(\Omega_{\epsilon})$ we have
$$
\mathcal H^2_g(\Sigma_{\tilde\epsilon})+\tilde \epsilon \mathcal H^3_g(\Omega_{\tilde\epsilon,\epsilon})\leq \mathcal H^2_g(\Sigma_{\epsilon}).
$$
We obtain \eqref{Eq: volume estimate} by combining above inequality with the fact $\mathcal H^2_g(\Sigma_{\tilde\epsilon})\geq A_g$ and  the area estimate for $\Sigma_\epsilon$ from Lemma \ref{Lem: area estimate}.
\end{proof}

Now we are ready to prove the rigidity part of Theorem \ref{Thm: main}.
\begin{proof}[Proof of rigidity]
As in the proof of Proposition \ref{Prop: sequence}, we fix a Euclidean ball $B_{r_0}$ such that there is a mean-convex foliation outside $B_{r_0}$. Let $\epsilon_0$ and $\Lambda_0$ be the same as in Lemma \ref{Lem: area estimate}. Fix a small $\epsilon\in (0,\epsilon_0)$. Through solving the $\mu$-bubble problem in the proof of Proposition \ref{Prop: sequence} and taking limit region as in the proof of Lemma \ref{Lem: volume estimate}, we can find a strictly outer-minimizing $2$-sphere $\Sigma_{\epsilon}$ contained in $B_{r_0}$ with constant mean curvature $\epsilon$ enclosing the origin $O$.

Fix a $\gamma\in(1,2)$. From Lemma \ref{Lem: volume estimate} we can construct another strictly outer-minimizing $2$-sphere $\Sigma_{\epsilon^\gamma}$ inside $\Sigma_{\epsilon}$ with constant mean curvature $\epsilon^\gamma$ such that we have
$$
\mathcal H^3_g(\Omega_{\epsilon^\gamma,\epsilon})\leq \Lambda_0 \epsilon^{2-\gamma}.
$$
Here and in the sequel $\Omega_{\tau,\delta}$ is always denoted to be the region bounded by $\Sigma_{\tau}$ and $\Sigma_\delta$. From induction we can construct a sequence of strictly outer-minimizing $2$-spheres $\left\{\Sigma_{\epsilon^{\gamma^k}}\right\}_{k\in \mathbf N}$ enclosing the origin $O$, which satisfy
\begin{itemize}
\item $\Sigma_{\epsilon^{\gamma^{k+1}}}$ is enclosed by $\Sigma_{\epsilon^{\gamma^k}}$;
\item $\Sigma_{\epsilon^{\gamma^k}}$ has constant mean curvature $\epsilon^{\gamma^k}$;
\item we have
$$
\mathcal H^3_g\left(\Omega_{\epsilon^{\gamma^{k+1}},\epsilon^{\gamma^k}}\right)\leq \Lambda_0\left(\epsilon^{\gamma^k}\right)^{2-\gamma}=\Lambda_0\epsilon^{\gamma^k(2-\gamma)}.
$$
\end{itemize}
As a consequence, we obtain
\begin{equation}\label{Eq: uniform area bound}
\mathcal H^3_g\left(\Omega_{\epsilon^{\gamma^k},\epsilon}\right)\leq \Lambda_0\sum_{i=1}^{k-1}\epsilon^{\gamma^i(2-\gamma)}<\Lambda_0\epsilon^{\gamma(2-\gamma)}\left(1-\epsilon^{(\gamma-1)(2-\gamma)}\right)^{-1},
\end{equation}
where we use the elementary inequality
$$
\gamma^i\geq 1+i(\gamma-1)\mbox{ for }i\in \mathbf N_+\mbox{ and }\gamma\in (1,2).
$$
Based on the uniform area bound \eqref{Eq: uniform area bound} we are able to show that all surfaces $\Sigma_{\epsilon^{\gamma^k}}$ intersect with a fixed compact subset $K$ independent of $k$. As before, we pick up a modified (signed) distance function $\rho_{\epsilon}:\mathbb R^3\setminus\{O\}\to (-\infty,+\infty)$ to the surface $\Sigma_\epsilon$ with $\lip \rho_\epsilon<1$, $\rho_{\epsilon}^{-1}(0)=\Sigma_\epsilon$ and $\rho_\epsilon<0$ inside $\Sigma_{\epsilon}$. Given any $T<0$ if the surface $\Sigma_{\epsilon^{\gamma^k}}$ does not intersect $\rho_{\epsilon}^{-1}\left([T,0]\right)$, then we can deduce
\[
\begin{split}
\mathcal H^3_g\left(\Omega_{\epsilon^{\gamma^k},\epsilon}\right)&\geq \int_{\rho_\epsilon^{-1}([T,0])}|\mathrm d\rho_\epsilon|_g\,\mathrm d\mathcal H^3_g\\
&=\int_T^0\mathrm ds\int_{\rho_{\epsilon}^{-1}(s)}\mathrm d\mathcal H^2_g\\
&\geq A_g|T|.
\end{split}
\]
Combining this with \eqref{Eq: uniform area bound} we conclude that the surface $\Sigma_{\epsilon^{\gamma^k}}$ intersects with the compact subset
$$
K=\rho_{\epsilon}^{-1}\left([T_0,0]\right)\mbox{ with }T_0=-\frac{\Lambda_0\epsilon^{\gamma(2-\gamma)}}{A_g\left(1-\epsilon^{(\gamma-1)(2-\gamma)}\right)}.
$$

Recall from Lemma \ref{Lem: area estimate} that we have $\mathcal H^2_g(\Sigma_{\epsilon^{\gamma^k}})\leq A_g+\Lambda_0\epsilon_0^2$. Since $\Sigma_{\epsilon^{\gamma^k}}$ are stable constant mean curvature surfaces of which the mean curvatures are uniformly bounded, it follows from \cite[Theorem 3.6]{ZZ2020} that we have a uniform curvature estimate for these surfaces $\Sigma_{\epsilon^{\gamma^k}}$ in any compact subset of $\mathbb R^3\setminus\{O\}$. Consequently, these surfaces converge smoothly to an area-minimizing surface $\Sigma_{min}$, whose area is no greater than $A_g$. In particular, the stability of $\Sigma_{min}$ yields
$$
\int_{\Sigma_{min}}|\nabla\phi|^2+K\phi^2\,\mathrm d\mathcal H^2_g\geq \frac{1}{2}\int_{\Sigma_{min}}(R(g)+|A|^2)\phi^2\,\mathrm d\mathcal H^2_g\geq 0
$$
for any $\phi\in C^\infty_0(\Sigma_{min})$, where $K$ is the Gaussian curvature of $\Sigma_{min}$ with the induced metric. It follows from \cite[Theorem 8.11]{GL1983} by Gromov and Lawson that $\Sigma_{min}$ is closed due to its finite area. Therefore, $\Sigma_{min}$ turns out to be a minimal $2$-sphere with area $A_g$ enclosing the origin $O$.

Through minimizing the area functional outside $\Sigma_{min}$ it is standard to find a strictly outer-minimizing surface $\Sigma_h$ homologous to $\Sigma_{min}$ with area $A_g$ that consists of minimal $2$-spheres. From the definition of $A_g$ it is clear that $\Sigma_{h}$ can only be one minimal $2$-sphere enclosing the origin $O$. From the work \cite{HI2001} we conclude that the closed region outside $\Sigma_h$ is isometric to the half Schwarzschild manifold with mass
$$
m=\sqrt{\frac{A_g}{16\pi}}>0.
$$
This completes the proof.
\end{proof}

\appendix
\section{Trumpet-like AF manifolds}\label{Sec: trumpet-like}
\begin{lemma}
There is a complete and rotationally symmetric metric $g$ on $\mathbb R^n\setminus\{O\}$ such that
\begin{itemize}
\item $g$ is AF around the infinity of $\mathbb R^n$;
\item the scalar curvature $R(g)$ is nonnegative;
\item for any $\rho>0$ the coordinate sphere $S_\rho=\{x\in\mathbb R^n:|x|=\rho\}$ is mean-convex with respect to the outward unit normal.
\end{itemize}
\end{lemma}
\begin{proof}
The idea is to bend the Schwarzschild metric such that one of its two ends is asymptotic to a cylinder. In particular, we take 
$$
u_1(r)=r^{\frac{2-n}{2}}\mbox{ and } u_2(r)=1+r^{2-n}\mbox{ with }r=|x|\in (0,+\infty),
$$
which are conformal factors of the cylindrical metric and the Schwarzschild metric respectively. Since we have
$$
u_1'(r)=\left(1-\frac{n}{2}\right)r^{-\frac{n}{2}}\mbox{ and }u_2'(r)=(2-n)r^{1-n},
$$
it follows from the fact $n\geq 3$ that we can find a positive constant $r_0$ such that $u_1'(r)>u_2'(r)$ for all $r\in(0,2r_0)$. Fix a smooth cut-off function 
$$\zeta:(0,+\infty)\to [0,1]$$ such that
$\zeta(t)\equiv 1$ when $0\leq t\leq r_0$, $\zeta(t)\equiv 0$ when $t\geq 2r_0$ and $\zeta'(t)\leq 0$. 

We define
$$
u_\alpha(r)=\alpha+u_1(r_0)-\int_{r}^{+\infty}[\zeta(s)u_1'(s)+(1-\zeta(s))u_2'(s)]\,\mathrm ds,
$$
where $\alpha$ is a positive constant to be determined later. Consider the metric 
$$
g=u_\alpha^{\frac{4}{n-2}}(\mathrm dr^2+r^2g_{\mathbb S^{n-1}}).
$$
Since we have $u_\alpha(r)\geq u_1(r)$ for $r\in (0,r_0)$ and $u_\alpha(r)\geq u_2(r)$ for $r\geq 2r_0$, the metric $g$ is complete. Moreover, it is rotationally symmetric from its definition. Next let us verify the desired properties one by one.

To see that $g$ is AF around the infinity of $\mathbb R^n$, we notice that $u_\alpha$ has the expression
$$
u_\alpha(r)=\alpha_0+r^{2-n}\mbox{ when }r\geq 2r_0
$$
for some positive constant $\alpha_0$. After changing a coordinate chart around the infinity the metric $g$ turns out to be some Schwarzschild metric, which is of course AF.

To verify the nonnegativity of $R(g)$, we compute
\begin{equation}\label{Eq: laplacian}
\begin{split}
\Delta u_\alpha&=u_{\alpha}''(r)+u_\alpha'(r)\cdot\frac{n-1}{r}\\
&=\zeta(r)\Delta u_1+(1-\zeta(r))\Delta u_2+\zeta'(r)\left(u_1'(r)-u_2'(r)\right).
\end{split}
\end{equation}
Since the cylindrical metric and Schwarzschild metric both have nonnegative scalar curvature, we see $\Delta u_1\leq 0$ and $\Delta u_2\leq 0$. Since $u_1'(r)>u_2'(r)$ holds in the support of $\zeta'(r)$ and we have $\zeta'\leq 0$, the third term in the second line of \eqref{Eq: laplacian} is also non-positive. This means $\Delta u_\alpha\leq 0$ and so the scalar curvature 
$$
R(g)=-\frac{4(n-1)}{n-2}u_\alpha^{-\frac{n+2}{n-2}}\Delta u_\alpha
$$
is nonnegative.

To verify the mean-convexity of the coordinate spheres $S_\rho$ we compute
\[
\left(u_\alpha^{\frac{2}{n-2}}r\right)'=u_\alpha^{\frac{4-n}{n-2}}\left(u_\alpha(r)+\frac{2}{n-2}ru_\alpha'(r)\right).
\]
When $0<r<r_0$ we have
$$
u_\alpha(r)+\frac{2}{n-2}ru_\alpha'(r)>\alpha+u_1(r)+\frac{2}{n-2}ru_1'(r)=\alpha>0.
$$
When $r>2r_0$ we have
$$
u_\alpha(r)+\frac{2}{n-2}ru_\alpha'(r)>\alpha+1-r^{2-n}
$$
which is positive if we choose $\alpha$ to be larger than $(2r_0)^{2-n}$.
When $r_0\leq r\leq 2r_0$ we have
$$
u_\alpha(r)+\frac{2}{n-2}ru_\alpha'(r)>\alpha-\frac{2}{n-2}\|ru_\alpha'(r)\|_{L^\infty([r_0,2r_0])}.
$$
A straightforward computation shows
\[
|ru_\alpha'(r)|\leq |ru_1'(r)|+|ru_2'(r)|\leq C(n)\left(r_0^{\frac{2-n}{2}}+r_0^{2-n}\right)
\]
for a universal constant $C(n)$. By choosing $\alpha$ large enough, we can guarantee
$$
\left(u_\alpha^{\frac{2}{n-2}}r\right)'(r)>0\mbox{ for all }r\in(0,+\infty)
$$
and consequencely all coordinate spheres $S_\rho$ are mean-convex with respect to the outward unit normal.
\end{proof}


\end{document}